\documentclass[12pt]{article}

\usepackage{amsmath}
\usepackage{amssymb}
\usepackage{amsthm}
\newtheorem{theorem}{Theorem}

\begin{document}

\begin{center}
{\bf\Large Can brains generate random numbers?}\\
\vspace{0.5cm}
 Va\v sek Chv\' atal (Concordia University, Montreal)\footnote{Canada Research Chair in Discrete Mathematics}\\
 Mark Goldsmith (Concordia University, Montreal)
\end{center}

\begin{center}
{\bf Abstract}
\end{center} 
\vspace{-0.3cm}
{\small 
Motivated by EEG recordings of normal brain activity, we construct
arbitrarily large McCulloch-Pitts neural networks that, without any
external input, make every subset of their neurons fire in some
iteration (and therefore in infinitely many iterations). 
}

\section{Introduction}\label{sec.intro}
Epilepsy is a group of neurologic conditions, the common and
fundamental characteristic of which is recurrent, unprovoked epileptic
seizures. These seizures are transient changes in attention or
behavior, often accompanied by convulsions; they result from
excessive, abnormal firing patterns of neurons that are located
predominantly in the cerebral cortex (the convoluted outer layer of
gray matter that covers each cerebral hemisphere). Such abnormal
paroxysmal activity is usually intermittent and
self-limited.~\cite[p.2] {EP97}.  The World Health Organization
reports~\cite{WHO} that there are over 50 million epilepsy sufferers
in the world today, 85\% of whom live in developing countries.

In attempts to study epilepsy, selected patients are monitored
continuously for days at a time. During these periods, EEG
(electroencephalogram) or ECoG (electrocorticogram) recordings are
made. EEG recordings come from placing multiple electrodes on the
scalp of the patient. ECoG recordings produce far more accurate data,
but they require invasive surgery to place a grid of electrodes
directly on the cortex. 

Neurologists specialized in epilepsy are trained to read EEG/ECoG
recordings, so that mere visual inspection allows them to tell with a
reasonable degree of accuracy when a seizure might have occured.
There are a number of different types of seizures; two major
categories are {\em partial seizures\/} (originating in a small group
of neurons, called a seizure focus, and spreading to other brain
regions) and {\em generalized seizures\/} (showing simultaneous
disruption of normal brain activity in both cerebral hemispheres from
the onset) ~\cite{Wes00}.  The classification~\cite{ILAE81} developed
by the International League Against Epilepsy in 1981 divides each of
these two categories into several subcategories, some of which are
divided into subsubcategories or even further and these different
types of seizures have different EEG
manifestations~\cite{Fis99,Nie99}.  One frequent occurrence is a
transition from an irregular, disorderly EEG before the seizure (the
pre-ictal state) to more organized sustained rhythm of spikes or sharp
waves during the seizure (the ictal state) ~\cite[Chapter
2]{BroHol04}.  (Some researchers refer to the pre-ictal EEG informally
as `chaotic' or `random' in contrast with the rigorous definitions of
these modifiers in mathematics, where --- roughly speaking ---
`chaotic' means `highly dependent on initial conditions' and `random'
means `unpredictable'.)

In July 2009, in a seminar held at Concordia, Nithum Thain asked
whether some initial configuration could cause Conway's Game of
Life~\cite{Gar70} to evolve in a way resembling a partial seizure,
proceeding from an erratic flutter of apparently unpredictable
patterns to sustained rhythmic changes that would begin in a small
part of the grid and gradually spread, synchronized, over a larger
area before subsiding to give way to the initial erratic mode. In the
discussion that followed, a variation emerged: Could a McCulloch-Pitts
neural network behave like this?

To define these networks, we need first the notion of a {\em linear
  threshold function.\/} This is a function $f:{\bf R}^n \rightarrow \{0,1\}$
such that, for some real numbers $w_1,w_2,\ldots,w_n$ (mnemonic for
``weights'') and $\theta$ (mnemonic for ``threshold''),
\[
f(x_1,x_2,\ldots,x_n)=
\left\{
\begin{array}{lll}
1 & \mbox {if } \sum_{j=1}^n  w_jx_j \ge \theta,\\
0 & \rule{0pt}{16pt}\mbox {otherwise.}
\end{array}
\right.
\]
The function is thought of as a neuron with zero-one signals
$x_1,x_2,\ldots,x_n$ received at its synapses from the axons of other
neurons; positive weights correspond to excitatory synapses and
negative weights correspond to inhibitory synapses;
$f(x_1,x_2,\ldots,x_n)=1$ means that the neuron, given signals
$x_1,x_2,\ldots,x_n$ at time $t$, will fire (send the signal `one')
along its axon after a synaptic delay at time $t+1$.

A {\em McCulloch-Pitts neural network\/}~\cite{McCPit43}, with
nonnegative integer parameters $p$ and $n$ such that $p<n$, is a
collection of linear threshold functions
\[
f_i:\{0,1\}^n \rightarrow \{0,1\} \;\;\;\;\;\;\;\;\; (i=p+1,p+2,\ldots ,n).
\]
Given any  zero-one vector $s_{p+1},s_{p+2},\ldots, s_n$ (the initial
state of the network) and 
any sequence $\xi_1,\xi_2,\ldots,\xi_p$ of $p$ functions
\[
\xi_r:{\bf N} \rightarrow \{0,1\}\;\;\;\;\;\;\;\;\;
(r=1,2,\ldots ,p),
\]
it computes a sequence $x_{p+1},x_{p+2},\ldots,x_n$ of $n-p$ functions
\[
x_i:{\bf N} \rightarrow \{0,1\} \;\;\;\;\;\;\;\;\;
(i=p+1,p+2,\ldots ,n.)
\]
This is done by setting $x_i(0)=s_i$ and, for all nonnegative integers
$t$,
\[
x_i(t+1)= f_i(\xi_1(t),\ldots,\xi_p(t),x_{p+1}(t),\ldots,x_n(t)).
\]
We think of variable $t$ as marking discrete time; each of the $n-p$
neurons $p+1,p+2,\ldots ,n$ may receive its signals from any of the
$n$ neurons; it receives a signal from neuron $j$ if and only if its
$w_{j}$ is nonzero. The bits $\xi_r(t)$ with $r=1,2,\ldots p$ and the
bits $x_i(t)$ with $i=p+1,p+2,\ldots ,n$ tell us which neurons are
firing at time $t$; firing or not firing of a neuron at time $t+1$
depends on the signals arriving to it at time $t$.  Neurons
$1,2,\ldots p$ receive no signals; they have no axons synapsing upon
them; McCulloch and Pitts call them `the peripheral afferents' of the
network. We will restrict ourselves to McCulloch-Pitts neural networks
with no such peripheral afferents; to put it differently, we will set
$p=0$.  The entire network is then a function 
$\Phi:\{0,1\}^n\rightarrow \{0,1\}^n$ defined by
\[
\Phi(x) = (f_1(x), f_2(x), \ldots, f_n(x)) 
\]
and $(x_1(t),x_2(t),\ldots x_n(t))=\Phi^t(s_1,s_2\ldots ,s_n)$ for
all nonnegative integers $t$.

The McCulloch-Pitts concept of an artificial neuron has been
generalized to allow firing at intensities on a continuous scale
rather than in the all-or-none way: instead of being a linear
threshold function, each $f_i:{\bf R}^n \rightarrow{\bf R}$ is now defined by 
\[
\textstyle{f_i(x_1,x_2,\ldots,x_n)=\varphi_i(\sum_{j=1}^n w_{ij}x_j)}
\]
with some `transfer function' $\varphi_i:{\bf R} \rightarrow{\bf R}$; the
McCulloch-Pitts original special case has $\varphi_i(s)=H(s-\theta_i)$, where
$H$ is the {\em Heaviside step function\/} defined by 
\[
H(d)=
\begin{cases}
1 &\mbox{when } d\ge 0,\\
0 &\mbox{when } d<0.
\end{cases}
\]
All such mathematical abstractions of biological neurons are only
crude approximations of their actual behaviour. More credible models
are {\em spiking neurons,\/} which were anticipated by
Lapicque~\cite{Lap07,Abb99} long before the mechanisms of the
generation of neuronal action potentials were known; later on, a
different model was proposed by Hodgkin and Huxley \cite{HH52} and
subsequently elaborated by many other researchers~\cite{Tuc88,
  RWRB97,KocSeg98,MaaBis01,GerKis02}.

We restrict ourselves to the McCulloch-Pitts model (with no peripheral
afferents) in all its simplicity: it played a seminal role in the
development of artificial neural networks~\cite{Roj96,SRK95}
and even today it is routinely referenced in medical literature
(\cite{Hut11,MGS11,Yus11} are just three of the more recent
citations). We do not pretend that our findings have any biological
significance, but we hope that they may serve as a template for
generalizations to more realistic models of the brain, such as
networks of spiking neurons.

\section{The question}\label{sec.que}
Could a McCulloch-Pitts neural network simulate a partial seizure?  We
have replaced this question by its easier variation: Could such a
network simulate the pre-ictal state of a brain?  To put it
differently, are there irregular, disorderly, apparently unpredictable
McCulloch-Pitts networks?  An essential prerequisite of every such
network $\Phi:\{0,1\}^n\rightarrow \{0,1\}^n$ is that, starting from any state
$s$ in its domain $\{0,1\}^n$, it eventually produce every state in
this domain as an element of the trajectory $s, \Phi(s), \Phi^2(s),
\ldots\;$.  This means that the {\em period of $\Phi$\/}, defined as
the smallest $t$ such that $\Phi^{t+1}(s)=s$ for some $s$ in its
domain, equals the size of the domain. This observation leads us to
ask an even easier question: Are there McCulloch-Pitts networks
$\Phi:\{0,1\}^n\rightarrow \{0,1\}^n$ with period $2^n$?  (The experimental
study~\cite{MBCHPR02} of trajectories in randomly generated
McCulloch-Pitts networks is to some extent related to this question;
lengths of state cycles in Boolean networks, cellular automata, and
other finite dynamical systems are often interpreted as a measure of
their computational power \cite{LegMaa05}.) The result reported here
is that the answer is affirmative for every positive integer $n$. (In
addition, we have found a number of properties that threshold
functions, and tuples of threshold functions, must possess in order to
define such networks. These will be the subject of a subsequent
paper.)

\section{The answer}\label{sec.ans}
Given a positive integer $n$, we define a mapping $\Phi_n:\{0,1\}^n\rightarrow
\{0,1\}^n$ by $\Phi_n(x_1,x_2,\ldots ,x_n)=(y_1,y_2,\ldots ,y_n)$
where, with $m$ the largest subscript such that $(x_1,x_2,\ldots
,x_m)$ is an alternating vector, $(0,1,0,1,\ldots)$ or 
$(1,0,1,0,\ldots)$,
\begin{equation}\label{eq.defphi}
y_i = 
\begin{cases}
\overline{x}_m &\mbox{when } 1\le i\le m,\\
x_i &\mbox{when } m<i\le n.\\
\end{cases}
\end{equation}
(Here, as usual, $\overline{x}$ with $x\in\{0,1\}$ denotes $1-x$.) For instance, 
\[
\begin{array}{l}
\Phi_4(0,0,0,0)=(1,0,0,0),\;\;\;\Phi_4(0,0,0,1)=(1,0,0,1),\\\Phi_4(0,0,1,0)=(1,0,1,0),\;\;\;\Phi_4(0,0,1,1)=(1,0,1,1),\\
\Phi_4(0,1,0,0)=(1,1,1,0),\;\;\;\Phi_4(0,1,0,1)=(0,0,0,0),\\\Phi_4(0,1,1,0)=(0,0,1,0),\;\;\;\Phi_4(0,1,1,1)=(0,0,1,1),\\
\Phi_4(1,0,0,0)=(1,1,0,0),\;\;\;\Phi_4(1,0,0,1)=(1,1,0,1),\\\Phi_4(1,0,1,0)=(1,1,1,1),\;\;\;\Phi_4(1,0,1,1)=(0,0,0,1),\\
\Phi_4(1,1,0,0)=(0,1,0,0),\;\;\;\Phi_4(1,1,0,1)=(0,1,0,1),\\\Phi_4(1,1,1,0)=(0,1,1,0),\;\;\;\Phi_4(1,1,1,1)=(0,1,1,1) 
\end{array}
\]
Note that the definition of $\Phi_n$ implies that $\Phi_n$ is antisymmetric in the sense of 
\begin{align}\label{eq.pr1}
&\mbox{$\Phi_n(\overline{x})=\overline{\Phi_n(x)}$ for all $x$ in $\{0,1\}^n$}
\end{align}
and that 
\begin{align}\label{eq.pr2}
&\Phi_1(0)=1
\end{align}
and that, when $n\ge 2$, 
\begin{align}\label{eq.pr3}
&\Phi_n(x_1,x_2,\ldots ,x_{n-1},0)=
\begin{cases}
(1,1,\ldots,1,1) \mbox{ if } (x_1,x_2,\ldots ,x_{n-1},0) \mbox{ is}\\ 
 \quad \mbox{ the alternating vector }(\ldots, 1,0,1,0),\\
(\Phi_{n-1}(x_1,x_2,\ldots ,x_{n-1}),0) \mbox{ otherwise.} \\
\end{cases}
\end{align}
\begin{theorem}\label{thm.main1}
The period of $\Phi_n$ is $2^n$. 
\end{theorem}
\begin{proof}
  Straightforward induction on $n$, using properties (\ref{eq.pr1}),
  (\ref{eq.pr2}), (\ref{eq.pr3}), proves a finer statement: The $2^n$
  vectors $\Phi^t_n(0,0,\ldots ,0)$ with $t=0,1,\ldots ,2^n-1$ are
  pairwise distinct and $\Phi^t_n(0,0,\ldots ,0)$ with $t=2^n-1$ is
  the alternating vector $(\ldots, 0,1,0,1)$.
\end{proof}
\begin{theorem}\label{thm.main2}
  For every positive integer $n$ there are linear threshold
  functions
\[
f_{n,i}:\{0,1\}^n \rightarrow \{0,1\} \;\;\;\;\;\;\;\;\;
(i=1,2,\ldots ,n)
\]
such that 
\begin{equation}\label{phimcp}
\mbox{$\Phi_n(x) = (f_{n,1}(x), f_{n,2}(x), \ldots, f_{n,n}(x))$ for all $x$ in $\{0,1\}^n$.} 
\end{equation}
\end{theorem}
\begin{proof}
  For every positive integer $n$ and for all $i=1, 2, \ldots, n$, we
  will construct weights $w_{n,i,\;j}$ ($j=1,2,\ldots ,n$) and threshold
  values $\theta_{n,i}$. Then we will define 
\[
f_{n,i}(x_1, \ldots, x_n) = \textstyle{H\left(\sum_{j=1}^nw_{n,i,\;j} x_j -\theta_{n,i}\right)}
\]
and prove that (\ref{phimcp}) is satisfied.

Our construction of $w_{n,i,\;j}$ and $\theta_{n,i}$ is recursive. To begin, we set
\[
w_{1,1,1}=-1,\; \theta_{1,1}=0; 
\]
for all integers $n$ greater than $1$, we set 
\begin{align*}
w_{n,n,\;j} & = \begin{cases}
1 &\mbox{if } j \not\equiv n \bmod 2,\\
-1 &\mbox{if } j \equiv n \bmod 2 \text{ and } j<n,\\
n-2 &\mbox{if } j=n,
\end{cases}\\
\theta_{n,n} & = \lfloor n/2 \rfloor,\\
&\\
w_{n,n-1,\;j} & = \begin{cases}
w_{n-1,n-1,\;j}  &\mbox{if } j\le n-2,\\
w_{n-1,n-1,\;j}+1 &\mbox{if } j=n-1,\\
-1 &\mbox{if } j=n,
\end{cases}\\
\theta_{n,n-1} & = \theta_{n-1,n-1}, 
\end{align*}
and, when $i=1, 2, \ldots, n-2$, 
\begin{align*}
w_{n,i,\;j} & = \begin{cases}
w_{n-1,i,\;j} + w_{n-2,i,\;j} &\mbox{if } j\le n-2,\\
w_{n-1,i,\;j} &\mbox{if } j=n-1,\\
-1 &\mbox{if } j=n,
\end{cases}\\
\theta_{n,i} & = \theta_{n-1,i} + \theta_{n-2,i}-1.
\end{align*}
Since the sequence $\Phi_1,\Phi_2,\Phi_3,\ldots$ is completely
determined by its properties (\ref{eq.pr1}), (\ref{eq.pr2}), (\ref{eq.pr3}), 
proving that (\ref{phimcp}) is satisfied reduces to proving that 
\begin{itemize}
\item[(i)] $f_{n,i}(\overline{x})=\overline{f_{n,i}(x)}$ for all $i=1,\ldots n$ and all $x$ in $\{0,1\}^n$,
\end{itemize}
observing that $f_{1,1}(0)=1$, and proving that
\begin{itemize}
\item[(ii)] if $x$ in $\{0,1\}^n$ is the alternating vector $(\ldots, 1,0,1,0)$,\\
then $f_{n,i}(x)=1$ for all $i=1,\ldots n$,
\item[(iii)] if $(x_1,\ldots ,x_{n-1},0)$ in $\{0,1\}^n$ is not the alternating vector $(\ldots, 1,0,1,0)$,\\
then $f_{n,n}(x_1,\ldots ,x_{n-1},0)=0$,
\item[(iv)] if $(x_1,\ldots ,x_{n-1},0)$ in $\{0,1\}^n$ is not the alternating vector $(\ldots, 1,0,1,0)$,\\
then $f_{n,i}(x_1,\ldots ,x_{n-1},0)=f_{n-1,i}(x_1,\ldots ,x_{n-1})$ for all $i=1,\ldots n-1$.
\end{itemize}
Straightforward, if a little tedious, induction on $n$ shows that 
\begin{align}
\sum_{j \not\equiv n \bmod 2} w_{n,i,\;j} &= \theta_{n,i}  &\mbox{ for all $i=1,\ldots n$,}\label{cleannoteq}\\
\sum_{j \equiv n \bmod 2} w_{n,i,\;j} &= \theta_{n,i} - 1 &\mbox{ for all $i=1,\ldots n$.}\label{cleaneq}
\end{align}
Summing up each pair of these equations, we conclude that
\[
\textstyle{\sum_{j=1}^n w_{n,i,\;j} = 2\theta_{n,i} - 1 \mbox{ for all $i=1,\ldots n$,}}
\]
which is easily seen to imply (i); equations (\ref{cleannoteq}) alone
imply directly (ii); proposition (iii) follows from the definitions of
$w_{n,n,\;j}$ and $\theta_{n,n}$.

In proving (iv), we will treat $i=n-1$ separately from $i\le n-2$.

To prove that $f_{n,n-1}(x_1,x_2,\ldots
,x_{n-1},0)=f_{n-1,n-1}(x_1,x_2\ldots ,x_{n-1})$ for all $(x_1,\ldots
,x_{n-1})$ in $\{0,1\}^{n-1}$ other than the alternating vector $(\ldots, 0,1,0,1)$, 
recall that
\[
\textstyle{\sum_{j=1}^{n-1}w_{n,n-1,\;j} x_j =\sum_{j=1}^{n-1}w_{n-1,n-1,\;j} x_j + x_{n-1}}
 \mbox{ and } \theta_{n,n-1} = \theta_{n-1,n-1},.
\] 
It follows that $f_{n,n-1}(x_1,\ldots ,x_{n-1},0)\ne
f_{n-1,n-1}(x_1,\ldots ,x_{n-1})$ if and only if $x_{n-1}=1$ and
$\sum_{j=1}^{n-1}w_{n-1,n-1,\;j} x_j + 1 = \theta_{n-1,n-1}$. Since
$w_{n-1,n-1,\;n-1} = n-3$ and $\theta_{n-1,n-1} = \lfloor (n-1)/2
\rfloor$, this means $\sum_{j=1}^{n-2}w_{n-1,n-1,\;j} x_j = -\lceil
(n-3)/2 \rceil$; since
\begin{align*}
w_{n-1,n-1,\;j} & = \begin{cases}
1 &\mbox{if } j \not\equiv n-1 \bmod 2,\\
-1 &\mbox{if } j \equiv n-1 \bmod 2 \text{ and } j<n-1,
\end{cases}
\end{align*}
this means further that $(x_1,\ldots ,x_{n-1})$ is the alternating vector $(\ldots, 0,1,0,1)$.

To prove that we have $f_{n,i}(x_1,\ldots
,x_{n-1},0)=f_{n-1,i}(x_1,\ldots ,x_{n-1})$ for all $i=1,\ldots , n-2$
and for all $(x_1,\ldots ,x_{n-1})$ in $\{0,1\}^{n-1}$ other than the
alternating vector $(\ldots, 0,1,0,1)$, we shall use induction on $n$.
In the induction step, we distinguish between two cases.

{\sc Case 1:} {\em $(x_1,\ldots ,x_{n-1})$ is the alternating vector
  $(\ldots, 1,0,1,0)$.}\\
In this case, we do not use the induction hypothesis. Equations
(\ref{cleaneq}) show that
\[
\textstyle{\sum_{j=1}^{n-1}} w_{n,i,\;j}x_j + w_{n,i,\;n} = \theta_{n,i}-1  \quad \quad \mbox{for all $i=1,\ldots n$;}
\]
since $w_{n,i,\;n}=-1$ for all $i=1,\ldots n-1$, it follows that 
\[
\textstyle{\sum_{j=1}^{n-1}} w_{n,i,\;j}x_j  = \theta_{n,i}  \quad \quad \mbox{for all $i=1,\ldots n-1$,}
\]
and so $f_{n,i}(x_1,\ldots ,x_{n-1},0)=1$ for all $i=1,\ldots ,
n-1$. By (ii) with $n-1$ in place of $n$, we have
$f_{n-1,i}(x_1,\ldots ,x_{n-1})=1$ for all $i=1,\ldots , n-1$.

{\sc Case 2:} {\em $(x_1,\ldots ,x_{n-1})$ is not the alternating
  vector $(\ldots, 1,0,1,0)$.}\\
In this case, consider an arbitrary $(x_1,\ldots ,x_{n-1},0)$ in
$\{0,1\}^n$ other than the alternating vector $(\ldots, 1,0,1,0)$.
The induction hypothesis guarantees (alone if $x_{n-1}=0$ and
combined with (i) if $x_{n-1}=1$) that $f_{n-1,i}(x_1,\ldots
,x_{n-1})=f_{n-2,i}(x_1,\ldots ,x_{n-2})$ for all $i=1,\ldots
n-2$. 

If $i\le n-2$ and $f_{n-1,i}(x_1,\ldots ,x_{n-1})=1$, then
$f_{n-2,i}(x_1,\ldots ,x_{n-2})=1$, and so
\begin{multline*}
\textstyle{\sum_{j=1}^{n-1} w_{n,i,\;j}x_j  = 
\sum_{j=1}^{n-1} w_{n-1,i,\;j}x_j  + \sum_{j=1}^{n-2} w_{n-2,i,\;j}x_j}\\  \ge \theta_{n-1,i} + \theta_{n-2,i} > \theta_{n,i}, 
\end{multline*}
which implies $f_{n,i}(x_1,\ldots ,x_{n-1},0)=1$. 

If $i\le n-2$ and $f_{n-1,i}(x_1,\ldots ,x_{n-1})=0$, then
$f_{n-2,i}(x_1,\ldots ,x_{n-2})=0$, and so
\begin{multline*}
\textstyle{\sum_{j=1}^{n-1} w_{n,i,\;j}x_j  = 
\sum_{j=1}^{n-1} w_{n-1,i,\;j}x_j  + \sum_{j=1}^{n-2} w_{n-2,i,\;j}x_j}\\  \le (\theta_{n-1,i}-1) + (\theta_{n-2,i}-1) < \theta_{n,i}, 
\end{multline*}
which implies $f_{n,i}(x_1,\ldots ,x_{n-1},0)=0$.
\end{proof}

Implicit in our proof of Theorem~\ref{thm.main1} is a simple way of
transforming each trajectory
\begin{equation}\label{eq.smalltr}
(0,0,\ldots ,0)\mapsto \Phi_n(0,0,\ldots
,0)\mapsto \ldots \Phi_n^{N-1}(0,0,\ldots ,0) 
\end{equation}
with $N=2^n$ into the trajectory
\begin{equation}\label{eq.bigtr}
(0,0,\ldots ,0)\mapsto \Phi_{n+1}(0,0,\ldots ,0)\mapsto
\ldots \Phi_{n+1}^{2N-1}(0,0,\ldots ,0): 
\end{equation}
First append $0$ as the last bit to each point of the trajectory
(\ref{eq.smalltr}) and let $T$ denote the resulting sequence of $2^n$
vectors $(x_1,\ldots ,x_{n},0)$ in $\{0,1\}^{n+1}$; then flip every bit
($0\leftrightarrow 1$) of every vector in $T$ and let $\overline{T}$
denote the resulting sequence of $2^n$ vectors $(x_1,\ldots
,x_{n},1)$ in $\{0,1\}^{n+1}$; the trajectory (\ref{eq.bigtr}) is the
concatenation $T\overline{T}$. For instance, if $n=3$, then
(\ref{eq.smalltr}) is
\begin{align*}
& (0,0,0)  \mapsto (1,0,0) \mapsto (1,1,0) \mapsto (0,1,0)
  \mapsto\\
&  (1,1,1) \mapsto (0,1,1) \mapsto (0,0,1) \mapsto (1,0,1),
\end{align*}
$T$ is
\begin{align*}
& (0,0,0,0)  \mapsto (1,0,0,0) \mapsto (1,1,0,0) \mapsto (0,1,0,0)
  \mapsto\\
&  (1,1,1,0) \mapsto (0,1,1,0) \mapsto (0,0,1,0) \mapsto (1,0,1,0),
\end{align*}
$\overline{T}$ is
\begin{align*}
&  (1,1,1,1) \mapsto (0,1,1,1) \mapsto (0,0,1,1) \mapsto  (1,0,1,1)
 \mapsto\\
& (0,0,0,1)  \mapsto (1,0,0,1) \mapsto (1,1,0,1) \mapsto  (0,1,0,1),
\end{align*}
and (\ref{eq.bigtr}) is 
\begin{align*}
& (0,0,0,0)  \mapsto (1,0,0,0) \mapsto (1,1,0,0) \mapsto (0,1,0,0)
  \mapsto\\
&  (1,1,1,0) \mapsto (0,1,1,0) \mapsto (0,0,1,0) \mapsto (1,0,1,0)
  \mapsto\\
&  (1,1,1,1) \mapsto (0,1,1,1) \mapsto (0,0,1,1) \mapsto  (1,0,1,1)
 \mapsto\\
& (0,0,0,1)  \mapsto (1,0,0,1) \mapsto (1,1,0,1) \mapsto  (0,1,0,1).
\end{align*}

It may be interesting to note that our $\Phi_n$ can be specified in
yet another way.  Every one-to-one mapping $r:\{0,1\}^n\rightarrow
\{0,1,\ldots 2^n-1\}$ generates a mapping $\Phi:\{0,1\}^n\rightarrow
\{0,1\}^n$ through the formula
\[
\Phi(x)=r^{-1}(r(x)+1 \bmod 2^n).
\]
We have, for $s=r^{-1}(0)$ and for all $t=0,1,\ldots 2^n-1$, 
\[
x=\Phi^t(s) \Leftrightarrow t=r(x)
\]
(this can be checked by straightforward induction on $t$); it follows
that $\Phi$ has period $2^n$. Our $\Phi_n$ is generated by the mapping
$r_n:\{0,1\}^n\rightarrow \{0,1,\ldots 2^n-1\}$ defined by $r_n(x_1,x_2,\ldots
x_{n})=\sum_{j=1}^{n}c_j2^{j-1}$ with
\[
c_j  = \begin{cases}
\lvert x_{j}-x_{j+1}\rvert &\mbox{when } 1\le j<n,\\
x_n &\mbox{when }  j=n.
\end{cases}
\]
For instance,
\[
\begin{array}{l}
r_4(0,0,0,0)=0,\;r_4(0,0,0,1)=12,\;r_4(0,0,1,0)=6,\;r_4(0,0,1,1)=10,\\
r_4(0,1,0,0)=3,\;r_4(0,1,0,1)=15,\;r_4(0,1,1,0)=5,\;r_4(0,1,1,1)=9,\\
r_4(1,0,0,0)=1,\;r_4(1,0,0,1)=13,\;r_4(1,0,1,0)=7,\;r_4(1,0,1,1)=11,\\
r_4(1,1,0,0)=2,\;r_4(1,1,0,1)=14,\;r_4(1,1,1,0)=4,\;r_4(1,1,1,1)=8
\end{array}
\]
This mapping $r_n$ is one-to-one for every $n$: from the integer
$r_n(x_1,x_2,\ldots x_{n})$, we can recover its binary encoding
$(c_1,c_2,\ldots c_{n})$, from which we can recover first $x_{n}$,
then $x_{n-1}$, and so on until $x_{1}$. 

To see that $r_n(\Phi_n(x))=r_n(x)+1 \bmod 2^n$, observe that (i) if
$x$ is the alternating vector $(\ldots, 0,1,0,1$), then $r_n(x)=2^n-1$
and (ii) for all other vectors $(x_1,x_2,\ldots ,x_{n})$ in
$\{0,1\}^n$, the largest subscript $m$ such that $(x_1,x_2,\ldots
,x_m)$ is an alternating vector equals the smallest subscript $m$ such
that $c_m=0$, in which case $r_n(x_1,x_2,\ldots
x_{n})+1=\sum_{j=1}^{n}d_j2^{j-1}$ with
\[
d_j  = \begin{cases}
0 &\mbox{when } 1\le j<m,\\
1 &\mbox{when }  j=m,\\
c_j &\mbox{when }  m<j\le n
\end{cases}
\]
and, with $(y_1,\ldots y_{n})$ defined by (\ref{eq.defphi}), we
have $r_n(y_1,\ldots y_{n})=\sum_{j=1}^{n}d_j2^{j-1}$.

\section{How many $n$-neuron McCulloch-Pitts\\ networks with period $2^n$ are
there?}

When $n=2$, there are just two such networks,
\[
(x_1,x_2)\mapsto (H(-x_2), H(x_1-1)) \;\mbox{ and }\;
(x_1,x_2)\mapsto (H(x_2-1), H(-x_1)) 
\]
with $H$ the Heaviside step function defined in
Section~\ref{sec.intro}. The first of these two networks is $\Phi_2$;
the second is produced from $\Phi_2$ by switching the two coordinates.

Next, let us consider $n=3$. By permuting the three coordinates,
$\Phi_3$ produces six distinct McCulloch-Pitts networks:
\begin{align*}
000 \rightarrow 100 \rightarrow 110 \rightarrow 010 \rightarrow 111 \rightarrow 011 \rightarrow 001 \rightarrow 101 \rightarrow 000,\\
000 \rightarrow 010 \rightarrow 110 \rightarrow 100 \rightarrow 111 \rightarrow 101 \rightarrow 001 \rightarrow 011 \rightarrow 000,\\
000 \rightarrow 001 \rightarrow 101 \rightarrow 100 \rightarrow 111 \rightarrow 110 \rightarrow 010 \rightarrow 011 \rightarrow 000,\\
000 \rightarrow 001 \rightarrow 011 \rightarrow 010 \rightarrow 111 \rightarrow 110 \rightarrow 100 \rightarrow 101 \rightarrow 000,\\
000 \rightarrow 010 \rightarrow 011 \rightarrow 001 \rightarrow 111 \rightarrow 101 \rightarrow 100 \rightarrow 110 \rightarrow 000,\\
000 \rightarrow 100 \rightarrow 101 \rightarrow 001 \rightarrow 111 \rightarrow 011 \rightarrow 010 \rightarrow 110 \rightarrow 000.
\end{align*}
(Here, we write $x_1x_2x_3$ for $(x_1,x_2,x_3)$ and we appeal to the
fact that an $n$-neuron McCulloch-Pitts network with period $2^n$ is
fully specified by its trajectory.) By flipping bits
($0\leftrightarrow 1$), these six McCulloch-Pitts networks produce an
additional eighteen McCulloch-Pitts networks: flipping the first bit
(and rotating the resulting trajectory to make $000$ its starting
point), we get
\begin{align*}
000 \rightarrow 010 \rightarrow 110 \rightarrow 011 \rightarrow 111 \rightarrow 101 \rightarrow 001 \rightarrow 100 \rightarrow 000,\\
000 \rightarrow 011 \rightarrow 001 \rightarrow 101 \rightarrow 111 \rightarrow 100 \rightarrow 110 \rightarrow 010 \rightarrow 000,\\
000 \rightarrow 011 \rightarrow 010 \rightarrow 110 \rightarrow 111 \rightarrow 100 \rightarrow 101 \rightarrow 001 \rightarrow 000,\\
000 \rightarrow 001 \rightarrow 100 \rightarrow 101 \rightarrow 111 \rightarrow 110 \rightarrow 011 \rightarrow 010 \rightarrow 000,\\
000 \rightarrow 010 \rightarrow 100 \rightarrow 110 \rightarrow 111 \rightarrow 101 \rightarrow 011 \rightarrow 001 \rightarrow 000,\\
000 \rightarrow 001 \rightarrow 101 \rightarrow 011 \rightarrow 111 \rightarrow 110 \rightarrow 010 \rightarrow 100 \rightarrow 000;
\end{align*}
flipping the second bit (and rotating the resulting trajectory to make
$000$ its starting point), we get
\begin{align*}
000 \rightarrow 101 \rightarrow 001 \rightarrow 011 \rightarrow 111 \rightarrow 010 \rightarrow 110 \rightarrow 100 \rightarrow 000,\\
000 \rightarrow 100 \rightarrow 110 \rightarrow 101 \rightarrow 111 \rightarrow 011 \rightarrow 001 \rightarrow 010 \rightarrow 000,\\
000 \rightarrow 001 \rightarrow 010 \rightarrow 011 \rightarrow 111 \rightarrow 110 \rightarrow 101 \rightarrow 100 \rightarrow 000,\\
000 \rightarrow 101 \rightarrow 100 \rightarrow 110 \rightarrow 111 \rightarrow 010 \rightarrow 011 \rightarrow 001 \rightarrow 000,\\
000 \rightarrow 001 \rightarrow 011 \rightarrow 101 \rightarrow 111 \rightarrow 110 \rightarrow 100 \rightarrow 010 \rightarrow 000,\\
000 \rightarrow 100 \rightarrow 010 \rightarrow 110 \rightarrow 111 \rightarrow 011 \rightarrow 101 \rightarrow 001 \rightarrow 000;
\end{align*}
flipping the third bit (and rotating the resulting trajectory to make
$000$ its starting point), we get
\begin{align*}
000 \rightarrow 100 \rightarrow 001 \rightarrow 101 \rightarrow 111 \rightarrow 011 \rightarrow 110 \rightarrow 010 \rightarrow 000,\\
000 \rightarrow 010 \rightarrow 001 \rightarrow 011 \rightarrow 111 \rightarrow 101 \rightarrow 110 \rightarrow 100 \rightarrow 000,\\
000 \rightarrow 100 \rightarrow 101 \rightarrow 110 \rightarrow 111 \rightarrow 011 \rightarrow 010 \rightarrow 001 \rightarrow 000,\\
000 \rightarrow 010 \rightarrow 011 \rightarrow 110 \rightarrow 111 \rightarrow 101 \rightarrow 100 \rightarrow 001 \rightarrow 000,\\
000 \rightarrow 110 \rightarrow 100 \rightarrow 101 \rightarrow 111 \rightarrow 001 \rightarrow 011 \rightarrow 010 \rightarrow 000,\\
000 \rightarrow 110 \rightarrow 010 \rightarrow 011 \rightarrow 111 \rightarrow 001 \rightarrow 101 \rightarrow 100 \rightarrow 000.
\end{align*}
Since $\Phi_3$ is antisymmetric, flipping two or three bits produces
no additional McCulloch-Pitts networks.  To summarize, $\Phi_3$
produces an isomorphism class of $24$ networks, where `isomorphism'
means any composition of permuting subscripts and flipping bits. The
McCulloch-Pitts network defined by
\[
(x_1,x_2,x_3)\mapsto 
(H(-x_1+x_2-x_3), H(-x_1-x_2-x_3+1), H(-x_1+x_2+x_3-1))
\]
does not belong to this class: its trajectory is 
\[
000 \mapsto 110 \mapsto 100 \mapsto 010 \mapsto 111 \mapsto 001 \mapsto 011 \mapsto 101 \mapsto 000.
\]
By permuting subscripts and flipping bits, this new network produces a
new isomorphism class of $24$ networks. Computer search shows that
there are no $3$-neuron McCulloch-Pitts networks with period $8$ other
than these $48$ networks in these $2$ isomorphism classes.

Additional computer search shows that there are precisely $9984$
distinct $4$-neuron McCulloch-Pitts networks with period $16$ and that
these networks come in $56$ distinct isomorphism classes.  It seems
that the number of isomorphism classes of $n$-neuron McCulloch-Pitts
networks with period $2^n$ grows rapidly with $n$.

\section{Pseudorandom number generators}

Our $\Phi_n$ has period $2^n$, as long as could possibly be, but it is
not quite irregular, disorderly, and apparently unpredictable. To
point out two of its blatant blemishes, let $y_i$ denote the $i$-th
bit of a vector $y$ in $\{0,1\}^n$.

Our recursive description of the trajectory
\[
(0,0,\ldots ,0)\mapsto \Phi_n(0,0,\ldots
,0)\mapsto \ldots \Phi_n^{N-1}(0,0,\ldots ,0) 
\]
with $N=2^n$ implies first that 
\[
(0,0,\ldots ,0)\ne \Phi_n(0,0,\ldots ,0), \;\;\; \Phi_n^{N-2}(0,0,\ldots ,0)\ne \Phi_n^{N-1}(0,0,\ldots ,0) 
\]
and then that 
\begin{equation}\label{eq.blemish1}
x_1 = \Phi_n(x)_1 \;\Rightarrow \; \Phi_n(x)_1\ne \Phi_n^2(x)_1.
\end{equation}
Our recursive description of the trajectory also implies that for all
$i=1,2,\ldots ,n$, the trajectory splits up into segments of length
$2^{i-1}$ so that the $i$-th bit of $\Phi_n^{t}(0,0,\ldots ,0)$ is
constant in each segment; it follows that
\begin{equation}\label{eq.blemish2}
  \mbox{ for all $i=2,3,\ldots ,n$, }\;\;\;
x_i \ne \Phi_n(x)_i \;\Rightarrow \; \Phi_n(x)_i= \Phi_n^2(x)_i.
\end{equation}
Properties (\ref{eq.blemish1}), (\ref{eq.blemish2}) make $\Phi_n$ far
from irregular, disorderly, and apparently unpredictable: in a random
permutation $\Phi$ of $\{0,1\}^n$, we would expect
$x_1=\Phi(x)_1=\Phi^2(x)_1$ about $25\%$ of the time and, for each 
$i=2,3,\ldots ,n$, we would expect $x_i\ne \Phi(x)_i\ne \Phi^2(x)_i$
about $25\%$ of the time.

By definition, no computable function $g:X\rightarrow X$ (where $X$ is a
finite set of numbers) can generate random numbers in the sequence $s,
g(s), g^2(s), \ldots$: if $g(x)$ can be computed, then it is not
unpredictable. (For an exposition of the concept of randomness, we
recommend~\cite{Gac09}.) Functions $g:X\rightarrow X$ that seem to generate
random numbers are called {\em pseudorandom number generators.\/} 
Since each vector in $\{0,1\}^n$ is a binary encoding of an $n$-bit
nonnegative integer, every mapping $\Phi:\{0,1\}^n\rightarrow \{0,1\}^n$
induces a mapping $\Phi^{\ast}:\{0,1,\ldots , 2^n-1\}\rightarrow \{0,1,\ldots
, 2^n-1\}$, and so every irregular, disorderly, apparently
unpredictable McCulloch-Pitts network $\Phi$ induces a pseudorandom
number generator $\Phi^{\ast}$.  Let us write
\[
X_n=\{k/2^n: k=0,1,\ldots ,2^n-1\}.
\]
Scaling down $\Phi^{\ast}$ by the factor of $2^n$, we get a mapping
$g_\Phi:X_n\rightarrow X_n$; for large values of $n$, this mapping approximates a
pseudorandom number generator $g:[0,1)\rightarrow [0,1)$. 

Long period alone does not suffice to make a pseudorandom number
generator acceptable; in order to be acceptable, it has to pass a
number of statistical tests. A number of these tests is commonly
agreed on; our favourite ones are implemented in the software library
TestU01 of L'Ecuyer and Simard\cite{EcuSim07,EcuSim09}. In particular,
TestU01 includes batteries of statistical tests for sequences of
uniform random numbers in the interval $[0,1)$. The least stringent of
them, {\tt SmallCrush}, consists of the following ten tests:
\begin{center}
{\footnotesize
\begin{tabular}{l}
\rule{0pt}{12pt}
\verb+ 1 smarsa_BirthdaySpacings+\\
\rule{0pt}{12pt}
\verb+ 2 sknuth_Collision+\\
\rule{0pt}{12pt}
\verb+ 3 sknuth_Gap+\\
\rule{0pt}{12pt}
\verb+ 4 sknuth_SimpPoker+\\
\rule{0pt}{12pt}
\verb+ 5 sknuth_CouponCollector+\\
\rule{0pt}{12pt}
\verb+ 6 sknuth_MaxOft+\\
\rule{0pt}{12pt}
\verb+ 7 svaria_WeightDistrib+\\
\rule{0pt}{12pt}
\verb+ 8 smarsa_MatrixRank+\\
\rule{0pt}{12pt}
\verb+ 9 sstring_HammingIndep+\\
\rule{0pt}{12pt}
\verb+10 swalk_RandomWalk1+\\
\end{tabular}
}
\end{center}
Each of these tests produces a number $p$; a typical range where the
test is considered passed is $0.001\le p \le 0.999$.

Is there a McCulloch-Pitts network $\Phi$ whose pseudorandom number
generator $g_\Phi$ passes all ten tests of {\tt SmallCrush}? Our
networks $\Phi_n$ fail all ten.

We conclude with an example of a $4$-neuron McCulloch-Pitts
network which has neither of the two properties (\ref{eq.blemish1}),
(\ref{eq.blemish2}). This network is defined by
\begin{align*}
(x_1,x_2,x_3,x_4) \mapsto & \;(H(-x_1-x_2-2x_3-x_4+2), H(-x_1-x_2+x_3+2x_4-1),\\
& \phantom{(}\;H( -x_1+2x_2-x_3+x_4-1), H(2x_1-x_2-x_3+x_4-1))
\end{align*}
and its trajectory is 
\begin{align*}
& (0,0,0,0) \mapsto (1,0,0,0) \mapsto (1,0,0,1) \mapsto (1,1,0,1) \mapsto \\
& (0,0,1,1) \mapsto (0,1,0,0) \mapsto (1,0,1,0) \mapsto (0,0,0,1) \mapsto \\
& (1,1,1,1) \mapsto (0,1,1,1) \mapsto (0,1,1,0) \mapsto (0,0,1,0) \mapsto \\
& (1,1,0,0) \mapsto (1,0,1,1) \mapsto (0,1,0,1) \mapsto (1,1,1,0) \mapsto \\ 
& (0,0,0,0) \mapsto (1,0,0,0) \mapsto \ldots \;\; .
\end{align*}
For each subscript $i=1,2,3$, the triples
$(x_i,\Phi(x)_i,\Phi^2(x)_i)$ run through the entire set $\{0,1\}^3$;
the triples $(x_4,\Phi(x)_4,\Phi^2(x)_4)$ miss only two values, which
are $(0,1,0)$ and $(1,0,1)$.

\bigskip

\begin{center}
{\bf Acknowledgments}
\end{center}

\noindent 
This research was undertaken, in part, thanks to funding from the
Canada Research Chairs Program. In addition to expressing their
gratitude for this support, the authors want to thank Mark\' eta
Vysko\v cilov\' a, MD for enlightening conversations about neurology.

\end{document}